\documentclass{amsart}
\usepackage{amsaddr}
\usepackage{fullpage}
\usepackage[latin1]{inputenc}
\usepackage{subfigure}
\usepackage[round]{natbib}
\usepackage{amssymb, amsmath, dsfont}
\usepackage{stmaryrd} 
\usepackage{comment}
\usepackage{mathdots}

\newtheorem{theorem}{Theorem}
\newtheorem{lemma}[theorem]{Lemma}

\newtheorem{corollary}[theorem]{Corollary}
\newtheorem{application}{Application}

\newcommand{\naturals}{\mathds{N}}

\newcommand{\reals}{\mathds{R}}
\newcommand{\bigO}{\mathcal{O}}
\newcommand{\smallo}{o}
\newcommand{\transp}[1]{#1^\intercal}
\newcommand{\trace}{\operatorname{Tr}}
\newcommand{\Id}{I}
\newcommand{\vect}[1]{\vec{#1}}
\newcommand{\gvect}[1]{\overset{{}_{\shortleftarrow}}{#1}}
\newcommand{\diag}[1]{\hat{#1}} 
\newcommand{\adj}{\operatorname{adj}}
\newcommand{\smat}[1]{\left( \begin{smallmatrix} #1 \end{smallmatrix} \right)}

\newcommand{\mn}{c}
\newcommand{\GRnm}{G_{R, \vect{r}}(n,m)}
\newcommand{\SGRnm}{\mathit{SG}_{R, \vect{r}}(n,m)}
\newcommand{\nmR}{$(n,m)$-$(R, \vect{r})$}
\newcommand{\Rr}{$(R,\vect{r})$}

\newcommand{\ff}{\Phi}
\newcommand{\gradf}{\gvect{\scriptstyle{\nabla}}_{\ff_{\mn}(\vphi)}}
\newcommand{\gradfx}{\gvect{\scriptstyle{\nabla}}_{\ff_{\mn}(\vx)}}

\newcommand{\Hf}{\mathcal{H}_{\ff_{\mn}(\vphi_{\mn})}}
\newcommand{\Hfx}{\mathcal{H}_{\ff_{\mn}(\vx)}}
\newcommand{\vphi}{\vect{\varphi}}
\newcommand{\vepsilon}{\vect{\epsilon}}
\newcommand{\vone}{\vect{\scriptstyle{1}}}
\newcommand{\vzero}{\vect{\scriptstyle{0}}}
\newcommand{\vtau}{\vect{\tau}}
\newcommand{\vgamma}{\vect{\gamma}}

\newcommand{\vT}{\vect{T}}
\newcommand{\vx}{\vect{x}}
\newcommand{\gvphi}{\gvect{\varphi}}
\newcommand{\gvepsilon}{\gvect{\epsilon}}
\newcommand{\gvone}{\gvect{\scriptstyle{1}}}
\newcommand{\gvzero}{\gvect{\scriptstyle{0}}}
\newcommand{\dphi}{\diag{\varphi}}
\newcommand{\gvtau}{\gvect{\tau}}
\newcommand{\gvgamma}{\gvect{\gamma}}

\newcommand{\gvT}{\gvect{T}}
\newcommand{\gvx}{\gvect{x}}
\newcommand{\dtau}{\diag{\tau}}
\newcommand{\dgamma}{\diag{\gamma}}
\newcommand{\dT}{\diag{T}}
\newcommand{\dx}{\diag{x}}
\newcommand{\weight}{\omega}
\newcommand{\edge}[1]{\overline{#1}}
\newcommand{\seqv}{\operatorname{seqv}}
\newcommand{\Edge}{\operatorname{edge}}

\newcommand{\UVnm}{G^{(U, V)}_{R, \vect{r}}(n,m)}
\newcommand{\onex}{\mathcal{S}}

\newcommand{\SV}{\mathit{SV}}
\newcommand{\vy}{\vect{y}}
\newcommand{\gvy}{\gvect{y}}

\newcommand{\Rcol}{R^{\operatorname{(col)}}}
\newcommand{\Rfriend}{R^{\operatorname{(fs)}}}

\begin{document}

\title{Enumeration and structure of inhomogeneous graphs}

\author{\'Elie de Panafieu}
\address{Research Institute for Symbolic Computation (RISC) \\
Johannes Kepler Universit\"at \\
Altenbergerstra\ss{}e 69 \\
A-4040 Linz, Austria}
\thanks{This work was partially founded by the ANR Magnum,
the Austrian Science Fund (FWF) grant F5004.}

\begin{abstract}
We analyze a general model
of weighted graphs,
introduced by de Panafieu and Ravelomanana (2014)
and similar to the \emph{inhomogeneous graph model} 
of S\"oderberg (2002).
Each vertex receives a \emph{type} 
among a set of $q$ possibilities
as well as a \emph{weight} corresponding to this type,
and each edge is weighted 
according to the types of the vertices it links.
The weight of the graph is then the product
of the weights of its vertices and edges.
We investigate the sum of the weights of such graphs
and prove that when the number of edges is small,
almost all of them contain no component with more than one cycle.
Those results allow us to give a new proof
in a more general setting
of a theorem of Wright (1961)
on the enumeration of properly colored graphs.
We also discuss applications related
to social networks.

\end{abstract}

\maketitle

The success of graphs relies on two contradictory properties.
They are simple enough to appear naturally in many applications,
but at the same time rich enough
to enjoy non-trivial mathematical properties.
Extensions such as hypergraphs and digraphs,
and various notions of randomness on those objects
have been proposed to address a wider range of applications.
In this article, we focus on a model of inhomogeneous graphs
introduced by \cite{PR14} 
and similar to the \emph{inhomogeneous graph model}
of \cite{S02}.
Our main tool is analytic combinatorics,
and this work has been influenced by the articles
of \cite{FKP89} 
and \cite{JKLP93}.

In Section~\ref{sec:model}, we present the model
and two applications.
The first one is a new proof of a result of \cite{W72}
on the enumeration of properly colored graphs,
the second is an enumeration of graphs
where each vertex receives a number of attributes
and two vertices can only be linked by an edge
if they share at least one common attribute.
Section~\ref{sec:global} provides theorems for
the enumeration of inhomogeneous graphs
with a given number of vertices~$n$ and edges~$m$.
The set of inhomogeneous graphs
that contain no component with more than one cycle
is analyzed in Section~\ref{sec:UV}.
We prove in Section~\ref{sec:unique_min}
that when $\frac{m}{n}$ is small enough,
almost all inhomogeneous graphs belong to this set,
and derive more explicit results than in Section~\ref{sec:global}
on the global enumeration of inhomogeneous graphs.
Section~\ref{sec:simple} extends the previous results
to inhomogeneous graphs without loops nor multiple edges.

    \section{Notations, models and applications} \label{sec:model}

    \subsection{Notations}

The row vector $(u_1, \ldots, u_q)$ is denoted by $\gvect{u}$,
and the column vector $\transp{(u_1, \ldots, u_q)}$ by $\vect{u}$.
The diagonal matrix $\diag{u}$ has main diagonal $\vect{u}$,
and $\vone$ is the vector with all coefficients equal to $1$.
%
We adopt the notation $\vect{u}^{\, \vect{v}}$ for the product $\prod_i u_i^{v_i}$.
The functions $\log$ and $\exp$ are applied coefficient-wise to the vectors,
\textit{i.e.} $\log(\gvect{u}) = ( \log(u_1), \ldots, \log(u_q) )$.
When $\sum_i n_i = n$, the multinomial notation $\binom{n}{\vect{n}}$ 
denotes $n! / \prod_i n_i !$.
The adjugate of a matrix $M$, equal to the transpose of the cofactor matrix, is $\adj(M)$.
Open intervals, closed intervals and integer intervals 
are denoted by $]x,y[$, $[x,y]$ and $[a..b]$.

    \subsection{Graph model}

The \emph{uniform graph model},
also called \emph{multigraph process}, 
has been studied using analytic combinatorics 
by \cite{FKP89} and \cite{JKLP93}.
This model produces a random 
vertex-labelled graph 
with $n$ vertices and $m$ edges
by drawing $2m$ vertices
$v_1 w_1 \ldots v_m w_m$
uniformly independently in $[1..n]$,
and adding to the graph the edges $\edge{v_i w_i}$
for $i$ from $1$ to $m$: 
\[
  \Edge(G) = \{ \edge{v_i w_i}\ |\ 1 \leq i \leq m\}.
\]
The graph is \emph{simple} if it contains neither loops nor multiple edges.
If the output of the process is conditioned to be simple, the model reduces 
to the classic $G(n,m)$ graph model of Erd\H{o}s and R\'enyi.
The number of ordered \emph{sequences of vertices} $v_1 w_1 \ldots v_m w_m$
that correspond to a graph $G$ is denoted by $\seqv(G)$
\[
  \seqv(G) = |
  \{ v_1 w_1 \ldots v_m w_m\ |\ 
    \{ \edge{v_i w_i}\ |\ 1 \leq i \leq m\} = \Edge(G) \} |.
\]
Observe that a graph~$G$ with $m$ edges is simple
if and only if 
its number of sequences of vertices $\seqv(G)$
is equal to $2^m m!$.
For this reason, \cite{JKLP93} introduced
the \emph{compensation factor} 
\[
  \kappa(G) = \frac{\seqv(G)}{2^m m!}.
\]
%
The \emph{number of graphs} in a family
is defined as the sum of their compensation factors, 
although this quantity needs not be an integer.
However, when the graphs are simple, 
the number of graphs is equal 
to the actual cardinality of the family.
For example, the total number of multigraphs
with $n$ vertices and $m$ edges is $\frac{n^{2m}}{2^m m!}$.

    \subsection{Inhomogeneous graph model}

The original \emph{inhomogeneous graph model} was introduced by \cite{S02}
as a generalization of the classic $G(n,p)$ random graph model,
and extended by \cite{BJR07}.
In this model, each vertex receives a \emph{type},
which is an integer in $[1..q]$,
and the probability that a vertex
of type~$i$ and one of type~$j$
are linked is the coefficient~$(i,j)$
of a symmetric $q \times q$ matrix $R$.
We consider in this paper a variant of this model,
introduced by \cite{PR14} and
closer to the uniform graph model.
Its parameters are an irreducible symmetric $q \times q$ matrix $R$
and a vector $\vect{r}$ of size $q$, 
both with non-negative coefficients.
We call \emph{inhomogeneous graph}, or \emph{\Rr-graph}, 
a labelled graph where
(loops and multiple edges are allowed) where
\begin{itemize}
\item
each vertex $v$ has a \emph{type} $t(v)$
which is an integer in $[1..q]$
and a \emph{weight} $r_{t(v)}$,
\item
each edge $\edge{vw}$ receives a weight $R_{t(v), t(w)}$.
\end{itemize}
The \emph{weight} $\omega(G)$ of an \Rr-graph $G$ 
is the product of
the compensation factor of the underlying graph
(which is equal to $1$ if the graph is simple),
the weights of the vertices
and the weights of the edges
\[
  \omega(G) = 
  \kappa(G) 
  \prod_{u \in G} r_{t(u)} 
  \prod_{\edge{vw} \in G} R_{t(v), t(w)}.
\]
One can also think of the parameters $(r_i)$ and $(R_{i,j})$
as variables marking the vertices and the edges
according to their types and the types of their ends.
We define the \emph{number of \Rr-graphs} in a family
as the sum of their weights.
This convention will be justified by the applications.
An \emph{\nmR-graph} is an \Rr-graph
with $n$ vertices and $m$ edges.
Let $n(G)$ denote the number 
of vertices of a graph $G$ 
and $\mathcal{F}$ be a family of \Rr-graphs,
then the \emph{generating function} $F(z)$ of $\mathcal{F}$
is defined by
\[
  F(z) = \sum_{G \in \mathcal{F}} \omega(G) \frac{z^{n(G)}}{n(G)!}.
\]
Observe that an \Rr-graph that contains 
an edge of weight zero
has weight zero, and thus
does not contribute to the number of \Rr-graphs.
The next lemma justifies our assumption for $R$ to be irreducible.

\begin{lemma}
  Let $G_{R,\vect{r}}^{*}$ denote the set of \Rr-graphs
  with non-zero weight.
  If the matrix $R$ is reducible,
  then there exist 
  a non-trivial partition $T_1 \uplus \cdots \uplus T_k = [1..q]$
  of the set of types,
  symmetric irreducible matrices $S_1, \ldots, S_k$
  and vectors $\vect{s}_1, \ldots, \vect{s}_k$
  such that $G_{R,\vect{r}}^{*}$ is in bijection 
  with the Cartesian product $G_{S_1,\vect{s}_1}^{*} \times \cdots \times G_{S_s,\vect{s}_k}^{*}$.
  Specifically, for any graph $G$ in $G_{R,\vect{r}}^{*}$\ , 
  \begin{itemize}
  \item
  for all $i \neq j$, there is no edge between 
  a vertex of type in $T_i$ and one of type in $T_j$,
  \item
  for all $i$, the graph induced by $G$
  on the vertices with types in $T_i$ 
  is in $G_{S_i,\vect{s}_i}^{*}$.
  \end{itemize}
\end{lemma}

\begin{proof}
  Let $(\vect{e}_1, \ldots, \vect{e}_q)$ denote the canonical basis of $\reals^q$.
  Since $R$ is reducible, there is a partition $T_1, \ldots, T_k$ of $[1..q]$
  such that the matrix of $R$ on the basis
  \[
    ( \vect{e}_{T_1[1]}, \vect{e}_{T_1[2]}, \ldots, \vect{e}_{T_2[1]}, \vect{e}_{T_2[2]}, \ldots )
  \]
  has a block-diagonal shape $\operatorname{diag}(S_1, \ldots, S_k)$.
  For each $i$, we set $\vect{s}_i = (r_{T_i[1]}, r_{T_i[2]}, \ldots)$.
  There can be no edge between types in $T_i$ and in $T_j$ for $i \neq j$
  because its weight would be $0$.
  Therefore, any component of $G_{R,\vect{r}}^{*}$
  with a vertex of type in $T_i$ has all its types in $T_i$.
  By construction, such a component is in $G_{S_i,\vect{s}_i}^{*}$.
\end{proof}

In the following, the matrix $R$ is therefore
always assumed to be irreducible.
In this paper, we analyze asymptotic properties
of \nmR-graphs when $n$ goes to infinity, $m$ is equal to $\mn n$,
and $R$, $\vect{r}$ and $\mn$ are fixed.
We look forward to applications that would require
to relax those conditions, to guide us 
toward the generalization of the model.

    \subsection{Applications}

Inhomogeneous graphs have been used in~\cite{PR14}
to analyze the phase transition of satisfiability problems.
We present two new applications.

In the \emph{properly $q$-colored} graphs,
each vertex has a color in $[1..q]$
and no edge links two vertices with the same one.
We give a new proof of \cite[Theorem~$3$]{W72} on their enumeration.
This result is not to be confused with 
an enumeration of \emph{$q$-colorable} graphs,
a problem expected to be much more difficult 
and addressed by \cite{AC08}.

\begin{application}
  Let $\Rcol$ denote the $q \times q$ matrix 
  with all coefficients equal to $1$,
  except on the diagonal where they are equal to $0$.
  Then the number of properly $q$-colored $(n,m)$-graphs
  is equal to the number of $(n,m)$-$(\Rcol,\vone)$-graphs.
  When $\frac{m}{n}$ is in a compact interval of $\reals_{>0}$, its asymptotics  is
  \[
    \frac{n^{2m}}{2^m m!}
    \left( 1 + \frac{2}{q-1} \frac{m}{n} \right)^{-\frac{q-1}{2}}
    \left( 1 - \frac{1}{q} \right)^m
    q^n
    + \smallo(1).
  \]
  For properly $q$-colored simple graphs, the previous asymptotics
  is replaced by 
  \[
    \frac{n^{2m}}{2^m m!}
    \left( 1 + \frac{2}{q-1} \frac{m}{n} \right)^{-\frac{q-1}{2}}
    \left( 1 - \frac{1}{q} \right)^m
    q^n
    \exp \big( \left( \frac{m}{n} \right)^2 \frac{q}{q-1} \big)
    + \smallo(1).
  \]
\end{application}

\begin{proof}
  Let us identify the types and the colors.
  If an $(\Rcol, \vone)$-graph is properly colored, 
  the product of the weights of its edges is $1$,
  otherwise it is $0$,
  hence the first assertion of the theorem.
  The eigenvalues of $\Rcol$ are $q-1$ and $-1$,
  with multiplicities $1$ and $q-1$.
  The asymptotics is then a direct application
  of Theorem~\ref{th:eigenvectorone}
  and the remark following Theorem~\ref{th:simple}.
\end{proof}

As a second example, we consider a world 
where each person has a number $k$ of topics of interest
among a set of size $t \geq 2k$,
and where two people can only become friend 
if they share at least one common topic of interest.
We call \emph{friendship graph}
the graph where each vertex represents a person
and each edge a friendship relation.
The graph is naturally assumed to be simple.
To analyze the friendship graphs, we introduce the following notations.
Let $\sigma$ denote a numbering of the subsets of size $k$ of $[1..t]$,
and $\Rfriend$ the adjacency matrix 
of the complement of the \emph{Kneser graph}
(see, for example, \cite{GR01}). 
This matrix of dimension $q = \binom{t}{k}$ is defined by
\[
  \Rfriend_{i,j} = 
  \begin{cases}
  1 & \text{if } | \sigma(i) \cap \sigma(j) | \geq 1,\\
  0 & \text{otherwise}.
  \end{cases}
\]

\begin{application}
The number of friendship graphs with $n$ people
and $m$ friendship relations is equal 
to the number of simple $(n,m)$-$(\Rfriend, \vone)$-graphs.
When $\frac{m}{n}$ is in a closed interval of $]0,\frac{1}{2}[$,
almost all friendship graphs contain no component with more than one cycle.
There is a value $\beta > \frac{1}{2}$ such that
when $\frac{m}{n}$ is in a compact interval of $]0, \beta[$,
the asymptotics of friendship graphs is
\[
  \frac{n^{2m}}{2^m m!}
  \left(
    \binom{t}{k} - \binom{t-k}{k}
  \right)^{m}
  \binom{t}{k}^{n-m}
  C
  + \smallo(1)
\]
where the value $C$, bounded with respect to $n$, is
\[
  C =
  \exp \left(
    - \frac{ \binom{t}{k} }
      { \binom{t}{k} - \binom{t-k}{k} }
    \frac{m}{n}
    \left( 1 + \frac{m}{n} \right)
  \right)
  \prod_{j=1}^k
  \left(
    1 
    - (-1)^j
    \frac{2 m}{n}
    \frac{ \binom{t-k-j}{k-j} }
      { \binom{t}{k} - \binom{t-k}{k} }
  \right)^{- \frac{1}{2} \left( \binom{t}{j} - \binom{t}{j-1} \right)}.
\]
\end{application}

\begin{proof}
Identifying each type $i$
with the set of topics of interest $\sigma(i)$,
the definition of the matrix $\Rfriend$
implies that the weight of a simple $(\Rfriend,\vone)$-graph
is $1$ if it is a friendship graph, and $0$ otherwise.
The spectrum of the Kneser graph is known, and available in \cite{DL98}.
The spectrum of its complement follows:
the dominant eigenvalue is $\binom{t}{k} - \binom{t-k}{k}$
and for all $j$ from $1$ to $k$, $(-1)^j \binom{t-k-j}{k-j}$
is an eigenvalue with multiplicity $\binom{t}{j} - \binom{t}{j-1}$.
The result is then a consequence of Theorem~\ref{th:eigenvectorone}
and the remark that follows Theorem~\ref{th:simple},
with parameters $\trace(\Rfriend) = \binom{t}{k}$
and $\trace((\Rfriend)^2) = \binom{t}{k} - \binom{t-k}{k}$.
\end{proof}

Inhomogeneous graphs can as well handle generalizations of the model.
For example, the weight of a friendship could be a real value, 
function of the number of common topics of interest.
The parameters of the asymptotics, although still computable,
will become less explicit.

    \section{Global enumeration} \label{sec:global}

In this section, we reduce the problem of deriving the asymptotics
of \nmR-graphs to the location of the minimums of a function
parameterized by $R$, $\vect{r}$ and $\frac{m}{n}$.
We start with an explicit formula.

\begin{theorem} \label{th:exact}
The number of \nmR-graphs is
\begin{equation} \label{eq:Gexact}
  \GRnm =
  \frac{1}{2^m m!}
  \sum_{\{ \vect{n} \in \naturals^q\ |\ \gvect{1} \vect{n} = n \}}
  \binom{n}{\vect{n}}
  \vect{r}^{\ \vect{n}} 
  \left( \gvect{n} R \vect{n} \right)^m.
\end{equation}
\end{theorem}

\begin{proof}
Let us consider a fixed partition of the set of labels $[1..n]$
into $q$ sets $V_1, \ldots, V_q$ of sizes $n_1, \ldots, n_q$,
and denote by $G(\vect{n},m)$ the set of \nmR-graphs
where the type of the vertices in $V_i$ is $i$.
Then
the number of graphs in $G(\vect{n}, m)$ is expressed
by summation over all sequences of vertices as
\[
  \sum_{G \in G(\vect{n},m)} \weight(G)
  =
  \frac{1}{2^m m!} 
  \prod_{i=1}^q r_i^{n_i}
  \sum_{v_1,w_1, \ldots, v_m, w_m \in [1..n]^{2m}} 
  \prod_{i=1}^m R_{t(v_i), t(w_i)}.
\]
Switching the sum and the product, the previous equation becomes
\[
  \sum_{G \in G(\vect{n},m)} \weight(G)
  =
  \frac{1}{2^m m!}
  \prod_{i=1}^q r_i^{n_i}
  \Bigg( \sum_{1 \leq i, j \leq q} n_i n_j R \Bigg)^m
  =  
  \vect{r}^{\ \vect{n}}
  \frac{(\gvect{n} R \vect{n})^m}{2^m m!}.
\]
Equation~\eqref{eq:Gexact} is obtained by summation 
over all possible partitions $V_1 \uplus \cdots \uplus V_q = [1..n]$.
\end{proof}

To obtain the asymptotics of $\GRnm$,
we will apply in the proof of Theorem~\ref{th:global}
a multivariate Laplace method.
This method requires to give to the previous expression
a more suitable shape.

\begin{lemma} \label{th:laplace_setting}
  Let $\onex$ denote the set 
  $\{ \vx \in \reals^q_{\geq 0}\ |\ \gvone \vx = 1\}$.
  The number of \nmR-graphs is
  \begin{equation} \label{eq:Glaplace}
    \GRnm =
    \frac{n^{2m}}{2^m m!}
    \frac{1}{(2 \pi n)^{\frac{q-1}{2}}}
    \sum_{\{\vect{n} \in \naturals^q\ |\ \gvone \vect{n} = n\}}
    A_n \left( \frac{\vect{n}}{n} \right)
    e^{- n \ff_{\frac{m}{n}} \left( \frac{\vect{n}}{n} \right)},
  \end{equation}
  where,  with the usual conventions $0 \log(0) = 0$ and $0^0 = 1$,
  the functions $A_n$ and $\ff_{\mn}$
  are defined on $\onex$ by
  \begin{align*}
    A_n(\vx) 
    &= 
    \frac{n!}{n^n e^{-n} \sqrt{2 \pi n}} 
    \prod_{i=1}^q
    \frac{(n x_i)^{n x_i} e^{-n x_i} \sqrt{2 \pi n x_i}}{\Gamma(n x_i + 1)} 
    \frac{1}{\sqrt{x_i}},
    \\
    \ff_\mn(\vx) 
    &=
    \left( \log(\gvx) - \log(\gvect{r}) \right) \vx
    - \mn \log \left( \gvx R \vx \right).
  \end{align*}
\end{lemma}

\begin{proof}
  We introduce in Expression~\eqref{eq:Gexact}
  the Stirling approximations
  of the factorials of the multinomial coefficient,
  and rescale each $n_i$ by a factor $1/n$
  \[
    \GRnm =
    \frac{n^{2m}}{2^m m!}
    \frac{1}{(2 \pi n)^{\frac{q-1}{2}}}
    \sum_{\substack{\vect{n} \in \naturals^q\\ \gvone \vect{n} = n}}
    \frac{n!}{n^n e^{-n} \sqrt{2 \pi n}}
    \prod_{i=1}^q
    \frac{n_i^{n_i} e^{-n_i} \sqrt{2 \pi n_i}}{n_i! \sqrt{n_i/n}}
    \left( 
      \frac{ \vect{r}^{\ \vect{n} / n} }
        { \left( \vect{n} / n \right)^{ \vect{n} / n } }
    \right)^n
    \left( \frac{\gvect{n}}{n} R \frac{\vect{n}}{n} \right)^m.
  \]
  The functions $A_n$ and $\ff_\mn$ are then introduced
  to simplify this expression.
\end{proof}

Let $E$ denote a $q \times (q-1)$ matrix 
with left-kernel of dimension $1$
containing the vector $\gvone$, \textit{e.g.}
\[ 
  E = 
  \smat{1 & 0 & \cdots & 0 \\
  0 & 1 & {\scriptstyle{\ddots}} & \vdots \\
  \vdots & \ddots & \ddots & 0 \\
  0 & \cdots & 0 & 1 \\
  -1 & \cdots & \cdots & -1}.
\]
Two vectors $\vect{u}$ and $\vect{v}$ belong to $\onex$
only if there is a vector $\vepsilon$ of dimension $q-1$
for which
$
  \vect{u} = \vect{v} + E \vepsilon.
$
The following lemma provides tools to locate
the minimums of the function $\ff_{\mn}$.

\begin{lemma} \label{th:beta}
For all $c > 0$, the Taylor expansion of $\ff_{\mn}$ near any point $\vx$
in the interior of $\onex$ is
\[
  \ff_{\mn}(\vx + E \vepsilon)
  =
  \ff_{\mn}(\vx)
  +
  \gradfx \vepsilon
  +
  \frac{1}{2}
  \gvepsilon \Hfx \vepsilon
  +
  \bigO(\| \epsilon \|^3),
\]
where the gradient vector and the Hessian matrix
have dimension $q-1$ and are defined by
\begin{align*}
  \gradfx &=
  \left(
  \log (\gvx)
  - \log (\gvect{r})
  - \frac{2\mn}{\gvx R \vx} \gvx R
  \right) E,
  \\
  \Hfx &=
  \transp{E}
  \left(
    \dx \phantom{}^{-1}
    +
    \frac{2\mn}{\gvx R \vx}
    \left(
      \frac{2}{\gvx R \vx}
      R \vx\, \gvx R
      - R
    \right)
  \right)
  E.
\end{align*}
If $\vphi$ is a minimum of $\ff_{\mn}$,
then $\vphi$ is in the interior of $\onex$, 
$\gradf = \gvzero$ and $\Hf$ is positive-semidefinite.
\end{lemma}

\begin{proof}
Let $\Psi_{\mn}$ denote the function
$
  \vx \to
  \left( \log(\gvx) - \log(\gvect{r}) \right) \vx
  - \mn \log \left( \gvx R \vx \right)
$
from $[0,1]^q \setminus \{\vzero\}$ to $\reals$.
Its restriction to $\onex$ is equal to $\ff_{\mn}$
and its Taylor expansion starts with
\[
  \Psi_\mn(\vx + \vepsilon)
  =
  \Psi_\mn(\vx)
  +
  \gvect{\nabla}_{\Psi_{\mn}(\vx)} \vepsilon
  +
  \frac{1}{2}
  \gvepsilon
  \mathcal{H}_{\Psi_{\mn}(\vx)}
  \vepsilon
  +
  \bigO(\| \vepsilon \|^3)
\]
where the gradient $\gvect{\nabla}_{\Psi_{\mn}(\vx)}$
and the Hessian matrix $\mathcal{H}_{\Psi_{\mn}(\vx)}$ 
of $\Psi_{\mn}$ are computed using partial derivations.
It follows that the Taylor expansion of $\ff_{\mn}$
near any point $\vx$ in the interior of $\onex$ is
\[
  \ff_\mn(\vx + E \vepsilon)
  =
  \ff_\mn(\vx)
  +
  \gvect{\nabla}_{\Psi_{\mn}(\vx)} E \vepsilon
  +
  \frac{1}{2}
  \gvepsilon \transp{E}
  \mathcal{H}_{\Psi_{\mn}(\vx)}
  E \vepsilon
  +
  \bigO(\| \vepsilon \|^3).
\]
Observing the limit of the gradient $\gradfx$ of $\ff_{\mn}$
when one coordinate of~$\vx$ vanishes,
we conclude that the local minimums of $\ff_{\mn}$
cannot be reached on the boundary of $\onex$,
and must therefore cancel $\gradfx$.
\end{proof}


Gathering the previous results, we can finally apply the multivariate Laplace method.

\begin{theorem} \label{th:global}
Let $[a,b]$ be a compact interval such that the function
$(c,\vx) \to \det(\Hfx)$ from $[a,b] \times \onex$ to $\reals$
does not vanish,
and let $\vphi_{\mn,1}, \ldots, \vphi_{\mn,s}$ denote the solutions
of $\gradf = \gvzero$ that correspond to local minimums of $\ff_{\mn}$, 
then when $c = \frac{m}{n}$ is in $[a,b]$, the asymptotics of \nmR-graphs is
\[
  \GRnm \sim
  \frac{n^{2m}}{2^m m!}
  \sum_{ \vphi \in \{ \vphi_{\mn,1}, \ldots, \vphi_{\mn, s} \} }
  \left( 
    \frac{ \vect{r}^{\, \vphi }}
    { \vphi^{\, \vphi }}
  \right)^n
  \frac{ \left( \gvphi R \vphi \right)^m }
  { \sqrt{\det(\mathcal{H}_{\ff_{\mn}(\vphi)}) \prod_{i=1}^q \varphi_i} }.
\]
\end{theorem}

\begin{proof} 
We inject in the integral representation of the sum~\eqref{eq:Glaplace}
the following relations
\[
A_n(\vphi + E \vepsilon) \sim
\prod_{i=1}^q \varphi_i^{-1/2}
+ \bigO(\|\vepsilon\|),
\ 
\quad 
e^{-n \ff_{\mn}(\vphi + E \vepsilon)} = 
\left(
\frac{ \vect{r}^{\, \vphi }}
{\vphi^{\, \vphi }} 
\right)^n
\left( \gvphi R \vphi \right)^m
e^{- \frac{1}{2} n \gvepsilon \Hf \vepsilon + \bigO(n \| \vepsilon \|^3)},
\]
valid for any minimum $\vphi$ of $\ff_{\mn}$,
and apply the multivariate Laplace method,
presented in \cite[Chapter~$5$]{PW13}).
\end{proof}

The previous theorem requires to locate the minimums of $\ff_{\mn}$,
and to avoid the values of $\mn$ for which those minimums cross or merge.
Even when the dimension of the matrix $R$ is $2$, 
those minimums can exhibit interesting behaviors.
For example, 
for $R = \smat{2&1\\1&2}$, 
$\ff_{\mn}$ has a unique minimum when $\mn \leq 1/6$
and two local minimums when $\mn > 1/6$.
It would be interesting to investigate if this transformation
in the analytic properties of $\ff_{\mn}$ matches
an evolution in the typical structure of \Rr-graphs.
Theorems~\ref{th:small_mn} and \ref{th:eigenvectorone} 
provide two examples of families of parameters $R$, $\vect{r}$ and $\frac{m}{n}$
for which $\ff_{\mn}$ has a unique minimum and 
a more explicit asymptotics for the number of \nmR-graphs can be derive.

    \section{Trees and unicycles} \label{sec:UV}

An \emph{\Rr-tree} is a connected \Rr-graph that contains no cycle.
Such a graph is \emph{rooted} if one vertex, 
called the \emph{root}, is marked.
An \emph{\Rr-unicycle} is a connected \Rr-graph
with exactly one cycle.
A classic result of \cite{ER60}
states that almost all $(n,m)$-graphs with $\frac{m}{n} < \frac{1}{2}$
contain only trees and unicycles. 
In this section, we derive a similar result for \Rr-graphs.
%
We also give a more explicit asymptotics for the number of \nmR-graphs
than in Theorem~\ref{th:global}
when $\frac{m}{n}$ is smaller than a value $\beta$, defined in Theorem~\ref{th:small_mn}.

\begin{lemma}
  Let $T_i(z)$, $U(z)$ and $V(z)$ denote the generating functions
  of \Rr-rooted trees with root of type $i$, \Rr-trees and \Rr-unicycles,
  and let $\vT(z)$ denote the vector $\transp{(T_1(z), \ldots, T_q(z))}$,
  then
  \[
    \vT(z)
    = z \diag{r} \exp \left( R \vT(z) \right), \quad
    U(z) = 
    \gvone \vT(z) - \frac{1}{2} \gvT(z) R \vT(z), \quad
    V(z) = 
    - \frac{1}{2} \log \left( \det \left( \Id - \dT(z) R \right) \right).
  \]
\end{lemma}

\begin{proof}
  An \Rr-rooted tree is a root with a set of sons which
  are \Rr-rooted trees themself. 
  Let $i$ denote the type of the root,
  and $j$ the type of the root of one of those sons,
  then the root has weight $r_i$ and
  the weight of the edge linking the root to the son
  is the coefficient $R_{i,j}$.
  Using the \emph{Symbolic Method}
  (see for example the book of \cite{FS09})
  the previous combinatorial description translates into 
  the first relation on $\vT(z)$.
  The expression for $U(z)$ is obtained
  using the \emph{Dissymmetry Theorem} presented by \cite{BLL97}.
  The proof of the expression of $V(z)$
  is a variation of~\cite[Proposition~$V.6$]{FS09}.
\end{proof}

\begin{lemma}
  The generating functions $\vT(z)$ 
  has the following singular expansions
  \[
    \vT(z) =
    \vtau 
    - \vgamma \sqrt{1 - z/\rho}
    + \bigO(1-z/\rho)
  \]
  where the value $\rho$ and the vectors $\vtau$ and $\vgamma$
  have positive coefficients and are characterized by 
  the system 
  \begin{equation} \label{eq:tau_rho_gamma}
    \vtau =
    \rho \diag{r} \exp \left( R \vtau \right),
    \quad
    \left( \Id - \diag{\tau} R \right) \vgamma = 0,
    \quad
    \frac{1}{2} \gvgamma R \dgamma R \vgamma = \gvone \vgamma.
  \end{equation}
\end{lemma}

\begin{proof}
  The square-root singular expansion of $\vT(z)$
  is a consequence of \cite[Proposition~$3$]{D99}.
  The constraints on its coefficients are obtained
  by injection of this expansion into the relation $\vT(z) = z \diag{r} \exp(R \vT(z))$
  and identification of the coefficients corresponding to the same power of $\sqrt{1-z/\rho}$.
  %
\end{proof}

We can now build \Rr-graphs
that contain only trees and unicycles.

\begin{theorem} \label{th:UV}
  We set $\alpha = \frac{1}{2} \frac{\gvtau R \vtau}{\gvone \vtau}$.
  For $\mn = \frac{m}{n}$ in any closed interval 
  of $] 0, \alpha [$, let $\zeta_{\mn}$ and $\vphi_{\mn}$ be characterized by 
  \begin{equation} \label{eq:UV_saddle_point}
    \frac{1}{2}
    \frac{\gvT(\zeta_{\mn}) R \vT(\zeta_{\mn})}
    {\gvone \vT(\zeta_{\mn})}
    =
    \mn
    \quad
    \text{and }
    \quad
    \vphi_{\mn} = \frac{\vT(\zeta_{\mn})}{\gvone \vT(\zeta_{\mn})},
  \end{equation}
  then the number of \nmR-graphs that contain
  only trees and unicycles is
  \begin{align}
    &\UVnm 
    \sim
    \frac{n^{2m}}{2^m m!}
    C_{\mn, \vphi_{\mn}}^{-1/2}
    \left(
      \frac{\vect{r}^{\, \vphi_{\mn}}}
        {\vphi_{\mn}^{\, \vphi_{\mn}}} 
    \right)^n
    \left( \gvphi_{\mn} R \gvphi_{\mn} \right)^m, \nonumber
  \\
  \text{where } \quad& 
  \label{eq:C}
    C_{\mn, \vx} =   
    \frac{1}{\mn}
    \left(
      \left(1-\mn \right) \gvone 
        \Big( \Id - \frac{2 \mn}{\gvx R \vx} \dx R \Big)^{-1} 
      \vx
      - 1
    \right)
    \det \left( \Id - \frac{2 \mn}{\gvx R \vx} \dx R \right).
  \end{align}
\end{theorem}

\begin{proof}
  A tree with $n$ vertices has $n-1$ edges,
  and a unicycle with $n$ vertices has $n$ edges.
  Therefore, an \nmR-graph that contains only trees and unicycles
  is a set of $n-m$ trees and a set of unicycles
  \[
    \UVnm =
    n! [z^n] \frac{U(z)^{n-m}}{(n-m)!} e^{V(z)}.
  \]
  We apply \cite[Theorem VIII.8]{FS09}
  to extract its asymptotics.
  The saddle-point equation is Equation~\eqref{eq:UV_saddle_point}.
  We then introduce the notation $\vphi$, which satisfies the relation
  \[
    \vT(\zeta_{\mn}) =
    \frac{2 \mn}{\gvphi_{\mn} R \vphi_{\mn}}
    \vphi_{\mn},
  \] 
  and apply Stirling approximations to rearrange the expression.
\end{proof}

In a longer version of this article,
we plan to enumerate connected \nmR-graphs
when $m-n \geq 1$ is fixed,
and to prove that such components
appear with a non-zero probability 
when $m = \alpha n + \bigO(n^{2/3})$.
This would extend to \Rr-graphs
the result obtained for graphs by \cite{JKLP93}
with $\alpha = \frac{1}{2}$.
Therefore, we conjecture that $\frac{m}{n} = \alpha$
is the threshold for the emergence of components 
with at least two cycles.
Following the approach of \cite{FKP89},
one could as well derive the limit law
of the number of edges when the first cycle appear
in a random \nmR-graph.

The following lemma  links the determinant of the Hessian matrix~$\Hfx$
to the value $C_{\mn, \vx}$. 

\begin{lemma} \label{th:detEHE_equal_C}
Let $\vtau$, $C_{\mn, \vx}$ and $\Hfx$
be defined by Equations~\eqref{eq:tau_rho_gamma},
\eqref{eq:C} and Lemma~\ref{th:beta}, and set $\alpha$ to $\frac{1}{2} \frac{\gvtau R \vtau}{\gvone \vtau}$,
then
for all $\mn \in [0,\alpha[$ and $\vx \in \onex$, the following identity holds
\[
  C_{\mn, \vx}
  =
  \det(\Hfx) \prod_{i=1}^q x_i.
\]
\end{lemma}

\begin{proof}
  We introduce the matrix $M = \frac{\gvx R \vx}{2 \mn} \dx \phantom{}^{-1} - R$
  and the vector $\vect{v} = \sqrt{ 2 / (\gvx R \vx) } \vx$.
  Since $E$ is a $q \times (q-1)$ matrix, 
  it becomes a square matrix $F$ if we concatenate to its right
  the column vector $\vect{v}$.
  The determinant of $\transp{F} M F$ can be expressed as $\det(F)^2 \det(M)$
  or using a block-determinant formula.
  The equality between those two expressions is
  \begin{equation} \label{eq:detEHE_equal_C}
    \frac{2}{\gvx R \vx} \det(M)
    =
    (\gvect{v} M \vect{v} + 1)
    \det(\transp{E} M E)
    -
    \det(\transp{E}(M + M \vect{v}\, \gvect{v} M) E).
  \end{equation}
  The properties of the matrix $E$ imply
  \[
    \det(\transp{E} M E) = \det(M) \gvone M^{-1} \vone
    \quad \text{and } \quad
    \transp{E}(M + M \vect{v}\, \gvect{v} M) E = \frac{\gvx R \vx}{2 \mn} \Hfx.
  \]
  The result is obtained by injection of those relations in Equation~\eqref{eq:detEHE_equal_C}
  and rearrangement of the terms.
\end{proof}

    \section{Global enumeration when $\ff_{\mn}$ has a unique minimum} \label{sec:unique_min}

In this section, we present two cases where
the result of Theorem~\ref{th:global}
can be made more specific.

\begin{theorem} \label{th:small_mn} 
Let $\beta$ denote the value 
$\sup \{ c\ |\ \forall \vx \in \onex,\ \det(\Hfx) > 0 \}$,
then $\beta$ is greater than $\alpha$.
The equation $\gradf = \gvzero$ defines implicitly on $[0, \beta[$ 
a unique solution $\vphi_{\mn}$,
given by Equation~\eqref{eq:UV_saddle_point}
when $\mn \in ]0,\alpha]$.
Finally, when $\frac{m}{n}$ is in a compact interval of $]0, \beta[$, 
the asymptotics of \nmR-graphs is
\[
  \GRnm
  \sim
  \frac{n^{2m}}{2^m m!}
  \left( 
  \frac{ \vect{r}^{\, \vphi_{\mn} }}
  { \vphi_{\mn}^{\, \vphi_{\mn} }}
  \right)^n
  \frac{ \left( \gvphi_{\mn} R \vphi_{\mn} \right)^m }
  { \sqrt{\det(\Hf) \det(\dphi_{\mn})} }.
\]
\end{theorem}

\begin{proof}
When $\mn = 0$, for all $\vx \in \onex$
the matrix $\mathcal{H}_{\ff_0(\vx)}$ is positive-definite.
By continuity of its eigenvalues with respect to $\mn$ and $\vx$,
$\Hfx$ stays positive-definite for all $\mn \in [0, \beta[$ and $\vx \in \onex$,
so the function $\ff_{\mn}$ is convex.
According to Lemma~\ref{th:beta},
$\ff_{\mn}$ has no minimum on the boundary of $\onex$,
so it has a unique minimum, which cancels its gradient $\gradfx$.
The asymptotics is then a consequence of Theorem~\ref{th:global}.
For $\mn \in ]0, \alpha[$, let us define
$\vphi_{\mn}$ as in Equation~\eqref{eq:UV_saddle_point}.
A direct computation shows that 
$\log(\gvphi_{\mn}) - \log(\gvect{r}) - 2 \mn \gvphi R / (\gvphi R \vphi)$
is collinear to $\gvone$, so $\vphi_{\mn}$ cancels $\gradfx$.
We extend continuously  $\vphi_{\mn}$ for $\mn = 0$ and $\mn = \alpha$ with
\[
  \vphi_0 
  = \lim_{z \to 0} \frac{\vT(z)}{\gvone \vT(z)}
  = \frac{\vect{r}}{\gvone \vect{r}},
  \qquad
  \vphi_{\alpha}
  = \lim_{z \to \rho} \frac{\vT(z)}{\gvone \vT(z)}
  = \frac{\vtau}{\gvone \vtau}.
\]
The last point we need to prove is that $\beta > \alpha$.
According to Lemma~\ref{th:detEHE_equal_C}, 
$\det(\Hfx)$ vanishes only if $C_{\mn, \vx}$ does,
and Theorem~\ref{th:UV} implies in particular $C_{\mn, \vphi_{\mn}} > 0$
when $\mn \in ]0, \alpha[$.
Observe that $\Hfx$ is independent of $\vect{r}$
and that for each $\vx \in \onex$,
there is a vector $\vect{r} \in \reals_{>0}^q$ such that $\vphi_{\mn} = \vx$.
This proves $\beta \geq \alpha$.
For $\mn = \alpha$,
$\vphi_{\alpha}$ is equal to $\vtau / (\gvone \vtau)$.
The definitions~\eqref{eq:tau_rho_gamma} of $\vtau$ 
and~\eqref{eq:C} of $C_{\mn, \vphi_{\mn}}$ then imply
\[
  \det \left( \Id - \frac{2 \alpha}{\gvphi_{\alpha} R \vphi_{\alpha}} \dphi_{\alpha} R \right) 
  = 0
  \quad \text{and } \quad
  C_{\alpha, \vphi_{\alpha}}
    =
    \frac{1-\alpha}{\alpha} \gvone \adj(\Id - \dtau R) \frac{\vtau}{\gvone \vtau}.
\]
The value of the generating function of unrooted \Rr-trees
$\gvone \vT(z) - \gvT(z) R \vT(z) / 2$
is positive at its dominant singularity $\rho$,
and $\vtau = \vT(\rho)$,
so $\alpha = \gvtau R \vtau / (2 \gvone \vtau)$ 
is smaller than $1$.
By definition of $\vtau$, the irreducible matrix $\dtau R$
has dominant eigenvalue $1$ of dimension $1$ with eigenvector $\vgamma$, 
defined in Equation~\eqref{eq:tau_rho_gamma}.
Therefore, $\adj(\Id - \vtau R)$ is proportional to $\vgamma\, \gvgamma$,
which has positive coefficients.
This implies $C_{\alpha, \vphi_{\alpha}} > 0$, 
so $\det(\mathcal{H}_{\ff_{\alpha}(\vx)})$
is positive for all $\vx$ in $\onex$ and $\ff_{\alpha}$ is still strictly convex.
Therefore, $\beta$ is greater than $\alpha$.
\end{proof}


The following corollary is obtained by comparison
of the asymptotics of $\UVnm$ from Theorem~\ref{th:UV}
and $\GRnm$ from Theorem~\ref{th:small_mn},
using the relation 
$C_{\mn, \vx} = \det(\Hfx) \prod_{i=1}^q x_i$ 
from Lemma~\ref{th:detEHE_equal_C}.

\begin{corollary}
  When $c = \frac{m}{n}$ is in a closed interval of $]0, \alpha[$,
  then almost all \nmR-graphs contain only trees and unicycles.
\end{corollary}

When $R$ is the adjacency matrix of a regular graph,
then $\vone$ is an eigenvector and $\vphi_{\mn}$ becomes explicit.

\begin{theorem} \label{th:eigenvectorone}
  Let $\lambda_1 \geq \cdots \geq \lambda_q$ denote the eigenvalues of $R$,
  and assume that $\vone$ is an eigenvector of $R$,
  then 
  when $\frac{m}{n}$ is in a compact interval of $]0, \beta[$,
  the asymptotics of $(n,m)$-$(R,\vone)$-graphs is
  \[
    G_{R, \vone}(n,m)
    \sim
    \frac{n^{2m}}{2^m m!}
    \lambda_1^m
    q^{n-m}
    \prod_{i=2}^q
    \left( 1 - 2 \frac{\lambda_i}{\lambda_1} \frac{m}{n} \right)^{-1/2}.
  \]
  If $\lambda_2$ is positive, $\beta < \frac{\lambda_1}{2\lambda_2}$,
  otherwise the previous asymptotics
  holds for $\frac{m}{n}$ in any compact interval of $\reals_{> 0}$.
\end{theorem}

\begin{proof}
$R$ has non-negative coefficients and is irreducible.
The Perron-Fr\"obenius theorem implies that its dominant eigenvalue
is positive and corresponds to the unique eigenvector with positive coefficients,
thus $R \vone = \lambda_1 \vone$.
The first assertion of the theorem is then a consequence of Theorem~\ref{th:small_mn}
with $\vphi_{\mn} = \vone / q$,
\[
  \gvect{\scriptstyle{\nabla}}_{\ff_{\mn}(\vone / q)} = \gvzero
  \quad \text{and } \quad
  \det(\mathcal{H}_{\ff_{\mn}(\vone/q)})
  =
  \frac{q^q}{1-2c}
  \det \left( \Id - \frac{2c}{\lambda_1} R \right)
  =
  q^q
  \prod_{i=2}^q
  \left( 1 - 2 c \frac{\lambda_i}{\lambda_1} \right).
\]
Observe that $\det(\mathcal{H}_{\ff_{\mn}(\vone/q)})$ vanishes at $\mn = \frac{\lambda_1}{2 \lambda_2}$.
Let us now consider the case $\lambda_2 \leq 0$.
The function $\log(\gvx) \vx$ 
reaches its unique minimum on $\onex$ at $\vone / q$.
To prove that the function
$\ff_{\mn}(\vx) = \log(\gvx) \vx - \mn \log(\gvx R \vx)$
does the same, it is then sufficient to prove that
$\gvx R \vx$ reaches at $\vone/q$ its global maximum.
%
%
%
Since $\vx$ is in $\onex$, there is a vector $\vy \in \reals^{q-1}$
that satisfies $\vx = \vone /q + E \vy$.
Since $\gvone E$ is equal to $\gvzero$, we obtain
\[
  \gvx R \vx = \frac{\lambda_1}{q} + \gvy \transp{E} R E \vy.
\]
The symmetry of $R$ implies the existence of an orthogonal matrix $Q$
with first row $\gvone/\sqrt{q}$ such that 
$R = \transp{Q} \operatorname{diag}(\lambda_1, \ldots, \lambda_q) Q$.
The first row of $Q E$ is $\gvzero$. 
Let $P$ denote the $(q-1) \times (q-1)$ matrix in the lower block of $Q E$.
Then $P$ is invertible, and 
\[
\transp{E} \transp{Q} \operatorname{diag}(\lambda_1, \ldots, \lambda_q) Q E 
= \transp{P} \operatorname{diag}(\lambda_2, \ldots, \lambda_q) P.
\]
We set $\vect{z} = P \vy$ and obtain 
$
  \gvx R \vx = \frac{\vone}{q} + \sum_{i=2}^q \lambda_i z_i^2.
$
Therefore, $\gvx R \vx$ reaches its global maximum $\lambda_1 / q$
when $\vect{z} = \vzero$, which implies $\vy = \vzero$ and $\vx = \vone/q$.
\end{proof}

A future direction of research 
is the expansion of the family of parameters $R$, $\vect{r}$ and $\mn$
for which we can link the asymptotics of \nmR-graphs
to the spectrum of $R$.
Information on the location of the minimums of $\ff_{\mn}$ 
for $\mn$ greater than $\beta$ would also be instructive.

    \section{Simple \Rr-Graphs} \label{sec:simple}

The previous sections focused on \Rr-graphs
where loops and multiple edges were allowed.
We now extend those results to simple \Rr-graphs,
starting with a theorem similar to Theorem~\ref{th:exact}.

\begin{theorem}
The number of simple \nmR-graphs is
\begin{equation} \label{eq:SGexact}
  \SGRnm =
  [w^m] \sum_{\gvone \vect{n} = n}
  \binom{n}{\vect{n}}
  \vect{r}^{\, \vect{n}}
  \prod_{1 \leq i < j \leq q}
    ( 1 + R_{i,j} w)^{n_i n_j}
  \prod_{i=1}^q
    (1 + R_{i,i} w)^{n_i(n_i-1) / 2}.
\end{equation}
\end{theorem}

\begin{proof}
We consider a partition $V_1 \uplus \cdots \uplus V_q = [1..n]$
of the set of vertices
and define $n_i = |V_i|$ for all~$i$.
Let $\mathit{SG}(\vect{n},m)$ denote the set of simple \nmR-graphs
where each vertex of $V_i$ has type $i$.
%
When $i \neq j$, there are $n_i n_j$ available edges between vertices of $V_i$ and $V_j$,
and for all $i$, there are $n_i (n_i-1)/2$ possible edges between vertices of $V_i$.
Therefore, the number of graphs 
in $\mathit{SG}(\vect{n}, m)$ is
\[
  \sum_{G \in \mathit{SG}(\vect{n},m)}
  \omega(G)
  =
  [w^m]
  \vect{r}^{\, \vect{n}}
  \prod_{1 \leq i < j \leq q}
    ( 1 + R_{i,j} w)^{n_i n_j}
  \prod_{i=1}^q
    (1 + R_{i,i} w)^{n_i(n_i-1) / 2}.
\]
The result of the Lemma is obtained by summation 
over all partitions $V_1 \uplus \cdots \uplus V_q = [1..n]$.
\end{proof}

%

\begin{theorem} \label{th:simple}
  With the notations of Theorem~\ref{th:global},
  for all $c=\frac{m}{n}$ in $[a,b]$,
  the asymptotics of simple \nmR-graphs is
  \[
    \SGRnm \sim
    \frac{n^{2m}}{2^m m!}
    \sum_{ \vphi \in \{ \vphi_{\mn,1}, \ldots, \vphi_{\mn, s} \} }
    \left( 
    \frac{ \vect{r}^{\, \vphi }}
    { \vphi^{\, \vphi }}
    \right)^n
    \left( \gvphi R \vphi \right)^m
    \frac{ e^{
          - \frac{\mn}{\gvphi R \vphi} \trace (\dphi R) 
          - \left( \frac{\mn}{\gvphi R \vphi} \right)^2 \trace ((\dphi R)^2) }}
    { \sqrt{\det(\mathcal{H}_{\ff_{\mn}(\vphi)}) \prod_{i=1}^q \varphi_i} }.
  \]
\end{theorem}

\begin{proof}
Starting with the exact expression~\eqref{eq:SGexact},
we replace $w$ with $n w$
\begin{align} \label{eq:SGR1}
  &\SGRnm =
  n^{m}
  \sum_{\gvone \vect{n} = n}
  \binom{n}{\vect{n}}
  \vect{r}^{\, \vect{n}}
  [w^m]
  e^{F_{n}(\vect{n}/n, w)},
  \\
  \text{where } \quad
  &
  F_{n}(\vx, w)
  =
  \log \Bigg(
  \prod_{1 \leq i < j \leq q}
    \left( 1 + R_{i,j} \frac{w}{n} \right)^{n x_i n x_j}
  \prod_{i=1}^q
    \left( 1 + R_{i,i} \frac{w}{n} \right)^{n x_i (n x_i-1) / 2}
  \Bigg). \nonumber
\end{align}
An expansion of the logarithm reduces this expression to
\[
  F_{n}(\vx, w)
  =
  n \gvx R \vx \frac{w}{2}
  - \frac{1}{2} \trace(\dx R) w
  - \frac{1}{4} \trace((\dx R)^2) w^2
  + \bigO(n^{-1}).
\]
With $\mn = \frac{m}{n}$ bounded, 
we apply \cite[Theorem~VIII.8]{FS09}
with saddle-point $\zeta = 2 \mn /(\gvx R \vx)$:
%
\[
  [w^m] e^{F_n(\vx,w)}
  =
  \frac{n^m}{2^m m!} 
  (\gvx R \vx)^m 
  \exp \left(
    - \frac{1}{2} \trace(\dx R) \frac{2 \mn}{\gvx R \vx}
    - \frac{1}{4} \trace((\dx R)^2) \left(\frac{2 \mn}{\gvx R \vx}\right)^2 
  \right)
  (1 + \bigO(n^{-1}))
\] 
holds uniformly for $\vx \in \onex$.
Adopting the notation $\ff_{\mn}$ of Lemma~\ref{th:laplace_setting},
Equation~\eqref{eq:SGR1} then becomes
\begin{align*}
  &\SGRnm =
  \frac{n^{2m}}{2^m m!}
  \frac{1}{(2 \pi n)^{\frac{q-1}{2}}}
  \sum_{\{\vect{n} \in \naturals^q\ |\ \gvone \vect{n} = n\}} 
  \mathit{SA}_n \left( \frac{\vect{n}}{n} \right)
  e^{- n \ff_{\frac{m}{n}} \left( \frac{\vect{n}}{n} \right)},
\\ \text{where } \quad & 
  \mathit{SA}_n(\vx) 
  = 
  \prod_{i=1}^q
  \frac{(n x_i)^{n x_i} e^{-n x_i} \sqrt{2 \pi n x_i}}{\Gamma(n x_i + 1)} 
  \frac{1}{\sqrt{x_i}}
  e^{
      - \trace(\dx R) \frac{\mn}{\gvx R \vx}
      - \trace((\dx R)^2) \left(\frac{\mn}{\gvx R \vx}\right)^2 
    }
  (1+\bigO(n^{-1})).
\end{align*}
The end of the proof is the same as in Theorem~\ref{th:global}.
\end{proof}

Theorems\ref{th:small_mn} and~\ref{th:eigenvectorone}
extend to simple \Rr-graphs in the same way.

\begin{corollary}
  When $\frac{m}{n}$ is in a closed interval of $]0, \alpha[$,
  almost all simple \nmR-graphs contain only trees and unicycles.
\end{corollary}

\begin{proof}
We verify that the asymptotics 
of simple \nmR-graphs containing only trees and unicycles
is equal to the asymptotics of all simple \nmR-graphs, derived in Theorem~\ref{th:simple}.
The generating function $U(z)$ of \Rr-trees is the same for graphs and simple graphs.
The generating function $\SV(z)$ of simple \Rr-unicycles
becomes
\[
  \SV(z) =
  V(z)
  - \frac{1}{2} \trace( \dT(z) R )
  - \frac{1}{4} \trace \left( (\dT(z) R)^2 \right)
\]
to avoid loops and double edges in the cycle.
The end of the proof is the same as in Theorem~\ref{th:UV}.
\end{proof}

\bibliographystyle{abbrvnat}
\bibliography{/home/elie/research/articles/bibliography/biblio}
\end{document}